\documentclass[10pt, letterpaper]{amsart}
\usepackage{amsmath}
\usepackage{amssymb,latexsym}
\usepackage{amscd}
\usepackage{xy} \xyoption{all} \CompileMatrices
\usepackage[
    bookmarksnumbered = true,
    pdftitle          = {Transfer maps and nonexistence of joint determinant},
    pdfsubject        = {{joint determinant}},
    pdfkeywords       = {{Milnor K-theory, transfer maps, Goodwillie group}},
    pdfauthor         = {{Sung Myung}},
    ]{hyperref}

\usepackage{amsthm}
    \newtheorem{theorem}    {Theorem}       [section]
    \newtheorem{lemma}      [theorem]       {Lemma}
    \newtheorem{corollary}  [theorem]       {Corollary}
    \newtheorem{prop}       [theorem]       {Proposition}
    
    \newtheorem{definition} [theorem]       {Definition}

\newcommand{\Z}{{\mathbb Z}}
\newcommand{\Q}{{\mathbb Q}}
\newcommand{\R}{{\mathbb R}}
\newcommand{\C}{{\mathbb C}}
\newcommand{\F}{{\mathbb F}}

\newcommand \M {{\mathcal M}}

\newcommand \Spec {{\operatorname{Spec }}}

\makeindex

\begin{document}

\title{Transfer maps and nonexistence of joint determinant}
         \author[Sung Myung]{Sung Myung}
         \address{Department of Mathematics Education, Inha University, 253 Yonghyun-dong, Nam-gu,
             Incheon, 402-751 Korea}
         \email{s-myung1\char`\@inha.ac.kr}


\begin{abstract}
Transfer Maps, sometimes called norm maps, for Milnor's $K$-theory were first defined by Bass and Tate (1972) for simple extensions of
fields via tame symbol and Weil's reciprocity law, but their functoriality had not been settled until Kato (1980). On the other
hand, functorial transfer maps for the Goodwillie group are easily defined. We show that these natural transfer maps actually agree
with the classical but difficult transfer maps by Bass and Tate. With this result, we build an isomorphism from
the Goodwillie groups to Milnor's $K$-groups of fields, which in turn provides a description of joint determinants for the commuting
invertible matrices. In particular, we explicitly determine certain joint determinants for the commuting invertible matrices over a finite field,
$\Q$, $\R$ and $\C$ into the respective group of units of given field.
\end{abstract}

        \maketitle

        \section{Introduction} \label{intro}
For a finite field extension $L/k$, we have the norm map $N_{L/k}: L^\times \rightarrow k^\times$.
In this article, we first show how to generalize the notion of the norm map to involve
$l$-tuples of elements $\alpha_1, \dots, \alpha_l \in L^\times$ or more generally $l$-tuples of commuting matrices $A_1, \dots, A_l \in GL_n(L)$.
These generalizations are called the transfer maps. For the case of $l$-tuples of elements in $L^\times$, the transfer map is defined
via Milnor's $K$-theory but it is somewhat difficult to describe because, although these elements give commuting $k$-linear automorphisms on $L$,
it is not easy to invent a process, which is `determinant' in case of a single element, to push them back to elements of $k$.
Bass and Tate (\cite{MR0442061}) were able to create a method for this case and its validity was verified by Kato (\cite{MR603953}).
For the case of $l$-tuples of commuting matrices in $GL_n(L)$, the transfer map is defined via the Goodwillie group. In this case,
we just observe that commuting matrices in $GL_n(L)$ give rise to commuting matrices in $GL_{dn}(k)$ by simply regarding $L$ as
$d$-dimensional $k$-vector space and declare that this process gives rise to the transfer maps.

We find that these two transfer maps are related and lots of ingredients involving determinants, matrix multiplications and
Kronecker products play roles in the dynamics of transfer maps. The transfer maps are the key ingredients connecting Milnor's $K$-theory
and the Goodwillie group. We remark that the Goodwillie group is one way to describe `motivic cohomology' of fields when the degree is
equal to the weight. Its natural generalization called Goodwillie-Lichtenbaum complex, after some modifications, gives the
motivic cohomology of arbitrary regular schemes in algebraic geometry. In this regard, Theorem \ref{Milnor-iso} may be viewed as
a reproduction of Nesterenko-Suslin theorem (\cite{MR992981}) which states that Milnor's $K$-theory and the motivic cohomology of fields
are isomorphic when the degree is equal to the weight. In fact, we show that $\displaystyle \bigoplus_{l \ge 0} K^M_l(k)$ and
$\displaystyle \bigoplus_{l \ge 0} GW_l(k)$ are isomorphic as graded rings with their respective product structures
which are described in Section \ref{milnor-K} and \ref{sec-Goodwillie}, respectively.

Finally, this result leads to the conclusion that there does not exist an interesting `determinant' $D$ for tuples of commuting invertible matrices over the usual fields
considered in Linear Algebra. More precisely, when $l \ge 2$ and $k$ is either $\Q$, $\C$ or a finite field,
there is no map $D$ of $l$-tuples of commuting matrices over $k$ onto $k^\times$ which satisfy the following 4 conditions.
When $k=\R$, we also require that $D$ is a continuous map when restricted to the set of commuting matrices in $GL_n(\R)$ for some $n$ with the usual topology. \\
$(i)$ (Multilinearity) For $l+1$ commuting invertible matrices $A_1, \dots, A_l$ and $B$ in $GL_n(k)$ for some $n \ge 1$, we have
$$D(A_1, \dots A_{i-1}, A_i B, A_{i+1}, \dots, A_l) = D(A_1, \dots, A_{i-1},A_i, A_{i+1} \dots, A_l) \cdot D(A_1, \dots,A_{i-1}, B, A_{i+1} \dots, A_l).$$
$(ii)$ (Block Diagonal Matrices) For commuting invertible matrices $A_1, \dots, A_l \in GL_m(k)$ and $B_1, \dots, B_l \in GL_n(k)$ for some
$m, n \ge 1$, we have
$$D\left(\begin{pmatrix} A_1 & 0 \\ 0 & B_1 \end{pmatrix}, \dots, \begin{pmatrix} A_l & 0 \\ 0 & B_l \end{pmatrix}\right)
= D(A_1, \dots, A_l) \cdot D(B_1, \dots, B_l).$$
$(iii)$ (Similar Matrices) For commuting matrices $A_1, \dots, A_l \in GL_n(k)$ and any $S \in GL_n(k)$,
we have $$D(S A_1 S^{-1}, \dots, S A_l S^{-1}) = D(A_1, \dots, A_l).$$
$(iv)$ (Polynomial Homotopy) For commuting $A_1(t), \dots, A_l(t) \in GL_n(k[t])$, we have $$D(A_1(0), \dots, A_l(0)) = D(A_1(1), \dots, A_l(1)).$$

The case where $k$ is either $\Q$, $\R$ or a finite field, respectively, is stated in
Corollary \ref{determinant-Q}, Corollary \ref{determinant-R} and Corollary \ref{determinant-finite-field}, respectively.
The case $k=\C$ is not stated as a separate corollary, but the nonexistence of `determinant' follows immediately from
Theorem \ref{Milnor-joint-determinant} and the fact that $K^M_l(\C)$ is uniquely divisible when $l \ge 2$.
More precise description is provided in Section \ref{sec-joint-det}.

       \section{Some computations in Milnor's $K$-groups} \label{milnor-K}
The basic reference for this section is \cite{MR0349811} and \cite{MR0442061}, but some computations later in the section are new
and are to be used later to prove the compatibility of transfer maps between the Milnor's $K$-groups and the Goodwillie groups.

For a field $k$, we define Milnor's $K$-groups $K^M_n(k)$ of $k$ as follow.

\begin{definition} \label{Milnor-symbol}
The $n$-th Milnor's $K$-group $K^M_n(k)$ is the additive quotient group of the tensor product
$(k^\times) ^{\otimes n}=k^\times \otimes k^\times \otimes \dots \otimes k^\times$ ($n$-times) by the
subgroup generated by the elements of the form $a_1 \otimes a_2 \otimes \dots \otimes a_n \in (k^\times) ^{\otimes n}$
where $a_i + a_j =1$ for some $1 \le i < j \le n$.
We denote by $\{a_1, a_2, \dots, a_n \}$, called a Milnor symbol,
the image of $a_1 \otimes a_2 \otimes \dots \otimes a_n \in (k^\times) ^{\otimes n}$ in $K^M_n(k)$.
\end{definition}

In particular, $K^M_1(k) \simeq k^\times$, but the group operation is written additively so that $\{a\} + \{b\} = \{ab\}$ and we set $K^M_0(k) = \Z$.
The following properties are some of basic relations for symbols in $K^M_2(k)$.
\begin{lemma} \label{Milnor-basic-relation}
Suppose that $a$, $b$ and $c$ are arbitrary nonzero elements of $k$. \\
(i) (Multilinearity) $\{ab, c\} = \{a,c\}+\{b,c\}$. In particular, $\{a, 1\} = 0$ for any $a \in k^\times$. \\
(ii) (Skew-symmetry) $\{a, b\} = -\{b,a\}$. \\
(iii) $\{a, -a\} = 0$.
\end{lemma}
\begin{proof}
$(i)$ Multilinearity is already part of the definition of tensor products.

$(iii)$ $ 0 = \{a, 1-a\} +\{a^{-1}, 1-a^{-1}\} = \{a, 1-a\} + \{a, \left(1-a^{-1} \right)^{-1} \}
= \{a, \left(1-a \right) \left(1-a^{-1} \right)^{-1} \} = \{a, -a\}$.

$(ii)$ $0 = \{ab, -ab \} = \{a, -ab \} + \{b, -ab\} = \{a, -a\} + \{a, b\} + \{b, a \} + \{b, -b\} = \{a, b\} + \{b, a \}$ by $(iii)$.

\end{proof}

Thanks to Lemma \ref{Milnor-basic-relation}, $\displaystyle \bigoplus_{l \ge 0} K^M_l(k)$ may be given an anti-commutative graded ring structure
by defining the product by the rule $\{a_1, \dots, a_p \} \cdot \{b_1, \dots, b_q \} = \{a_1, \dots, a_p, b_1, \dots, b_q \}$.
In particular, suppose that $1 \le n \le l$ and $a_1, \dots, a_n$ are nonzero elements of $k$ such that
$\{a_1, \dots, a_n\}=0$ in $K^M_n(k)$. Then, we have $\{a_1, \dots, a_n, b_{n+1}, \dots, b_l\}=0$ in $K^M_l(k)$ for arbitrary nonzero elements
$b_{n+1}, \dots, b_l$ of $k$.

\begin{lemma} \label{Milnor-relation2}
Let $c$ and $d$ be arbitrary nonzero elements of a field $k$. Then, we have the following relations in $K^M_2(k)$.\\
(i) $\displaystyle \{c, d \} = \{ {\frac c d}, d-c \} + \{-1, d\}$. \\
(ii) $\displaystyle \{c, d \} = \{ -{\frac c d}, d+c \}$.
\end{lemma}
\begin{proof}
$(i)$ $\displaystyle 0=\{ {\frac c d}, 1- {\frac c d} \} = \{ {\frac c d}, {\frac {d-c} d} \}
= \{{\frac c d}, d-c\}-\{c,d\}+\{d,d\}$. But, $\{d,d\} = \{-1, d\} + \{-d, d\} = \{-1, d\}$ by Lemma \ref{Milnor-basic-relation} $(iii)$.\\
$(ii)$ follows quickly from $(i)$ by substituting $-c$ for $c$ in the equality.
\end{proof}

Now let us prove the following key relation which will be used later in the proof of Lemma \ref{Weil-like-formula}.

\begin{prop} \label{Milnor-key-relation}
Suppose that $l \ge 1$ and that $l+1$ elements $x_0, x_1, x_2, \dots, x_l$ of a field $k$ is such that $x_i - x_j \ne 0$, whenever $i \ne j$ modulo
$l+1$, where the indices are considered modulo $l+1$. Then we have, in $\ K^M_l (k)$,
$$ \sum_{i=0}^l (-1)^{l(i+1)} \{x_{i+1}-x_i, x_{i+2}-x_{i}, x_{i+3}-x_i, \dots, x_{i+l}-x_i \} = \{-1, -1, -1, \dots, -1\}.$$
\end{prop}

\begin{proof}
We proceed by induction on $l$.
The case $l=1$ is straightforward since $\displaystyle -\{x_1-x_0\}+\{x_0-x_1\} = \{ {\frac {x_0-x_1} {x_1-x_0}} \} = \{-1\}$ in $K_1(k)$.

To prove the proposition we aim to prove by induction the following equality
\begin{multline} \label{key-1a}
\{ x_1-x_0, x_2-x_0, x_3-x_0, x_4-x_0, \dots, x_l-x_0 \} \\
 = \{x_0-x_1, x_2 - x_1, x_3-x_1, x_4-x_1, \dots, x_l-x_1 \} \\
 - \{x_0-x_2, x_1 - x_2, x_3-x_2, x_4-x_2, \dots, x_l-x_2 \} \\
 + \{x_0-x_3, x_1-x_3, x_2-x_3, x_4-x_3, \dots, x_l-x_3 \} \\
 - \dots +(-1)^{i+1} \{x_0-x_i, x_1-x_i, x_2-x_i, \dots, x_{i-1}-x_i, x_{i+1}-x_i, \dots, x_l-x_i \} + \cdots \\
 +(-1)^{l+1} \{x_0-x_l, x_1-x_l, x_2-x_l, x_3-x_l, x_4-x_l, \dots, x_{l-1}-x_l\} \\
 + \{-1, -1, -1, \dots, -1 \}.
\end{multline}
Once the relation (\ref{key-1a}) is proved, we deduce by Lemma \ref{Milnor-basic-relation} $(ii)$ that
 \begin{multline*}
\{ x_1-x_0, x_2-x_0, x_3-x_0, x_4-x_0, \dots, x_l-x_0 \} + (-1)^l \{x_2- x_1, x_3 - x_1, x_4-x_1,  \dots, x_l-x_1, x_0 - x_1 \} \\
+ \{x_3-x_2, x_4 - x_2, x_5-x_2, \dots, x_1-x_2 \}
+ (-1)^l \{x_4-x_3, x_5-x_3, x_6-x_3, \dots, x_2-x_3 \} \\
+ \dots +\{x_0-x_l, x_1-x_l, x_2-x_l, \dots, x_{l-1}-x_l \} = \{-1, -1, -1, \dots, -1\}
\end{multline*}
and the proof of the proposition will be complete if we multiply both sides by $(-1)^l$ since $\{-1, -1, \dots, -1\}$ is 2-torsion.

By the inductive hypothesis on $l-1$, we have
\begin{multline} \label{key-1b}
\{ x_1-x_0, x_2-x_0, x_3-x_0, x_4-x_0, \dots, x_{l-1}-x_0 \} \\
 = \{x_0-x_1, x_2 - x_1, x_3-x_1, x_4-x_1, \dots, x_{l-1}-x_1 \} \\
 - \{x_0-x_2, x_1 - x_2, x_3-x_2, x_4-x_2, \dots, x_{l-1}-x_2 \} \\
 + \{x_0-x_3, x_1-x_3, x_2-x_3, x_4-x_3, \dots, x_{l-1}-x_3 \} \\
 - \dots + (-1)^{i+1} \{x_0-x_i, x_1-x_i, x_2-x_i, x_3-x_i, \dots, x_{i-1}-x_i, x_{i+1}-x_i,\dots, x_{l-1}-x_i \} + \cdots \\
 +(-1)^{l} \{x_0-x_{l-1}, x_1-x_{l-1}, x_2-x_{l-1}, x_3-x_{l-1}, x_4-x_{l-1}, \dots, x_{l-2}-x_{l-1}\} \\
 + \{-1, -1, -1, \dots, -1 \}.
\end{multline}

Multiply on the right by $x_l - x_0$ and we get
\begin{multline*}
\{ x_1-x_0, x_2-x_0, x_3-x_0, x_4-x_0, \dots, x_{l-1}-x_0, x_l-x_0 \} \\
 = \{x_0-x_1, x_2-x_1, x_3-x_1, x_4-x_1, \dots, x_{l-1}-x_1, x_l-x_0 \} \\
 - \{x_0-x_2, x_1-x_2, x_3-x_2, x_4-x_2, \dots, x_{l-1}-x_2, x_l-x_0 \} \\
 + \{x_0-x_3, x_1-x_3, x_2-x_3, x_4-x_3, \dots, x_{l-1}-x_3, x_l-x_0 \} \\
 - \dots + (-1)^{i+1} \{x_0-x_i, x_1-x_i, x_2-x_i, x_3-x_i, \dots, x_{i-1}-x_i, x_{i+1}-x_i,\dots, x_{l-1}-x_i, x_l-x_0 \} + \cdots \\
 +(-1)^{l} \{x_0-x_{l-1}, x_1-x_{l-1}, x_2-x_{l-1}, x_3-x_{l-1}, x_4-x_{l-1}, \dots, x_{l-2}-x_{l-1}, x_l-x_0\} \\
 + \{-1, -1, -1, \dots, -1, x_l-x_0 \}.
\end{multline*}

Applying Lemma \ref{Milnor-relation2} $(ii)$ to the first and the last coordinates, each of the terms
$$\{x_0-x_i, x_1-x_i, x_2-x_i, x_3-x_i, \dots, x_{i-1}-x_i, x_{i+1}-x_i,\dots, x_{l-1}-x_i, x_l-x_0 \}$$
($i=1,2, \dots, l-1$) in the right hand side can be rewritten as
\begin{align*}
& \{ {\frac {x_0-x_i} {x_0-x_l}}, x_1-x_i, x_2-x_i, \dots, x_{i-1}-x_i, x_{i+1}-x_i,\dots, x_{l-1}-x_i, x_l-x_i \} \\
& \quad = \{x_0-x_i, x_1-x_i, x_2-x_i, \dots, x_{i-1}-x_i, x_{i+1}-x_i,\dots, x_{l-1}-x_i, x_l-x_i \} \\
& \quad \quad- \{x_0-x_l, x_1-x_i, x_2-x_i, \dots, x_{i-1}-x_i, x_{i+1}-x_i,\dots, x_{l-1}-x_i, x_l-x_i \}.
\end{align*}
Hence, we have
\begin{multline*}
\{ x_1-x_0, x_2-x_0, x_3-x_0, \dots, x_{l-1}-x_0, x_l-x_0 \} \\
= \sum_{i=1}^{l-1} (-1)^{i+1} \{x_0-x_i, x_1-x_i, x_2-x_i, \dots, x_{i-1}-x_i, x_{i+1}-x_i,\dots, x_{l-1}-x_i, x_l-x_i \} \\
-\sum_{i=1}^{l-1} (-1)^{i+1} \{x_0-x_l, x_1-x_i, x_2-x_i, \dots, x_{i-1}-x_i, x_{i+1}-x_i,\dots, x_{l-1}-x_i, x_l-x_i \} \\
+ \{-1, -1, -1, \dots, -1, x_l-x_0 \}.
\end{multline*}

Therefore, to prove (\ref{key-1a}) by induction on $l$, it suffices to show the equality
\begin{multline*}
-\sum_{i=1}^{l-1} (-1)^{i+1} \{x_0-x_l, x_1-x_i, x_2-x_i, \dots, x_{i-1}-x_i, x_{i+1}-x_i,\dots, x_{l-1}-x_i, x_l-x_i \} \\
+ \{-1, -1, -1, \dots, -1, x_l-x_0 \} \\
= (-1)^{l+1} \{x_0-x_l, x_1-x_l, x_2-x_l, \dots, x_{l-1}-x_l \} \\
+ \{-1, -1, -1, \dots, -1, -1 \},
\end{multline*}
or equivalently, as any Milnor symbol having $-1$ as a coordinate is 2-torsion,
\begin{multline*}
\sum_{i=1}^{l} (-1)^{i+1} \{x_0-x_l, x_1-x_i, x_2-x_i, \dots, x_{i-1}-x_i, x_{i+1}-x_i,\dots, x_{l-1}-x_i, x_l-x_i \} \\
 =  \{x_0-x_l, -1, -1, -1, \dots, -1 \}. \\
\end{multline*}

But, this equality is obtained by replacing $x_0, x_1, \dots, x_{l-1}$ with $x_1, x_2, \dots, x_{l}$ in the inductive hypothesis (\ref{key-1b}) as
\begin{multline*}
\{ x_2-x_1, x_3-x_1, x_4-x_1, \dots, x_l-x_1 \} \\
 = \{x_1-x_2, x_3 - x_2, x_4-x_2, \dots, x_l-x_2 \} \\
 - \{x_1-x_3, x_2 - x_3, x_4-x_3, \dots, x_l-x_3 \} \\
 + \{x_1-x_4, x_2 - x_4, x_3-x_4, \dots, x_l-x_4 \} \\
 - \dots +(-1)^{i+1} \{x_1-x_i, x_2 - x_i, \dots, x_{i-1}-x_i, x_{i+1}-x_i, \dots, x_l-x_i \} + \cdots \\
 +(-1)^{l} \{x_1-x_l, x_2 - x_l, x_3-x_l, x_4-x_l, \dots, x_{l-1}-x_l \} \\
 + \{-1, -1, -1, \dots, -1 \}.
 \end{multline*}
and then multiplying on the left by $x_0-x_l$.
\end{proof}

\section{Transfer maps for Milnor's $K$-groups} \label{sec-transfer-Milnor}

Let us recall a definition of the transfer map for Milnor's $K$-groups.
Let $k$ be any field and let $K=k(X)$ be the field of rational functions of the projective line ${\mathbb{P}}^1_k$ over $k$.
A discrete valuation $v$ of the field $K=k(X)$, which vanishes on $k$, is called a discrete valuation of $K/k$. For each discrete valuation $v$
of $K/k$, let $\pi_v$ be a uniformizing parameter
and $k_v =R_v/(\pi_v)$ be the residue field of the valuation ring $R_v=\{ r \in K | v(r) \ge 0 \}$.

There are two types of discrete valuations of $K/k$. For each monic irreducible polynomial $\pi_v$ of
the polynomial ring $k[X]$, we can associate a unique valuation $v$ of $K$ such that $v(\pi_v)=1$. The other type is $v_\infty$
where $\pi_v = 1/X$, i.e., $v(f) = -\deg f(X)$ whenever $f(X)$ is a polynomial in $k[X]$.

The higher tame symbol $\partial_v : K^M_{l+1}(K) \rightarrow K^M_l(k_v)$ is defined as follows (See \cite{MR0442061} or \cite{MR603953} for more details).
\begin{definition}
The tame symbol $\partial_v : K^M_{l+1}(K) \rightarrow K^M_l(k_v)$ is the unique epimorphism such that
$$\partial_v ( \{ u_1, \dots, u_l, y \}) = v(y) \{\overline{u_1}, \dots, \overline{u_l} \}$$
whenever $u_1, \dots, u_l$ are units of the valuation ring $R_v$.
\end{definition}

Let $v_\infty$ be the discrete valuation of $K=k(X)$, which vanishes on $k$, such that $v_\infty (X) = -1$. In particular, $v_\infty (f) = -\deg(f)$
if $f \in k[X]$ is a polynomial of degree $\deg(f)$.
Every simple algebraic extension $L$ of $k$ is isomorphic to $k_v$ for some discrete valuation $v \ne v_\infty$ which corresponds to a
prime ideal $\mathfrak{p}$ of $k[X]$. Conversely, every discrete valuation $v \ne v_\infty$ of $K/k$ gives rise to a simple algebraic extension
$L$ of $k$. This fact motivates the following definition of the transfer maps for simple extensions, which is due to Bass and Tate (\cite{MR0442061}).

\begin{definition} \label{transfer-Milnor}
The transfer maps $N_{k_v/k} : K^M_l(k_v) \rightarrow K^M_l(k)$ are the unique homomorphisms such that, for every $w \in K^M_{l+1}(k(X))$,
$\displaystyle \sum_{v} N_{k_v/k} \left( \partial_v w \right)=0$
where the sum is taken over all discrete valuations of $k(X)/k$ including $v_\infty$ on $k(X)$.
For $v=v_\infty$, we take $N_{v_\infty} = \rm{Id}$.
\end{definition}

Kato (\cite{MR603953} \S 1.7) has shown that these maps, if defined as compositions of transfer maps for simple extensions for a given tower of
simple extensions, depend only on the field extension $L/k$, i.e., that it enjoys functoriality. See also \cite{MR689382}.
Therefore, the transfer map is well-defined for arbitrary finite extension $L/k$ and functorial for compositions of finite field extensions.

We would like to mention some useful lemmas before we go onto the next section.

\begin{lemma} \label{generator-residue-field}
Let $v$ be the discrete valuation of $K=k(X)$ associated with a monic irreducible polynomial $\pi_v$ of degree $d$.
Denote by $\alpha$ the image of $X$ in the residue field $k_v = k[X]/(\pi_v)$.
Then the group $K^M_l(k_v)$ is generated by the symbols of the form $\{a_1, a_2, \dots, a_{l-r}, f_1(\alpha), f_2(\alpha), \dots, f_r(\alpha) \}$,
where $a_1, a_2, \dots, a_{l-r}$ are elements of $k^\times$ and $f_1, \dots, f_r$ are monic irreducible polynomials in $k[X]$
with $0 < \deg f_1 < \deg f_2 < \dots < \deg f_r < d$.
\end{lemma}
\begin{proof}
Any element $w \in K^M_l(k_v)$ can be written as
$w=\{g_1(\alpha), \dots, g_l(\alpha)\}=\partial_v( \{g_1, \dots, g_l, \pi_v \} )$ for some polynomials
$g_1, \dots, g_l \in k[X]$ of degree less than $d$.
So, it suffices to prove that if $\{g_1, \dots, g_l\}$ is a symbol in $K^M_l(K)$ where $g_1, \dots, g_l \in k[X]$ are of degree
less than $d$, then it can be written as a sum of symbols of the form
$\{a_1, a_2, \dots, a_{l-r}, f_1, f_2, \dots, f_r \}$
where $a_1, a_2, \dots, a_{l-r}$ are elements of $k^\times$ and $f_1, \dots, f_r$ are monic irreducible polynomials in $k[X]$
with $0 < \deg f_1 < \deg f_2 < \dots < \deg f_r < d$.

First of all, by multilinearity of Milnor symbols, we may suppose that $g_1, \dots, g_l$ are monic irreducible polynomials in $k[X]$.

Now let us assume that $f$ and $g$ are monic irreducible polynomials of the same degree, say, $m > 0$.
Write $f = g + h$ where $\deg h$ is less than $m$. In case $f=g$, $\{f,g\} = \{-f, f\} + \{-1, f\} = \{-1, f\}$ in $K^M_2(K)$.
In other cases, $h \ne 0$ and we have $\displaystyle {\frac g f} + {\frac h f }=1$.
Hence $\displaystyle \{ {\frac h f}, {\frac g f} \} = 0$. If we expand out the symbol using the multilinearity, then
we get $\{f,g\} = \{h,g\} - \{h,f\} + \{f,f\} = \{h,g\} - \{h,f\} + \{-1,f\}$. In both cases, we can always rewrite $\{f,g\}$
as a sum of symbols of the form $\{\phi , \psi \}$ where $\phi$ and $\psi$ are polynomials of $k[X]$ with $\deg \phi < \deg \psi \le m$.

The proof follows inductively from this observation and we are done.
\end{proof}

\begin{lemma} \label{norm-inductive-formula}
(Inductive formula for $N_{k_v/k}$)
For a generator $x=\{a_1, a_2, \dots, a_{l-r}, f_1(\alpha), f_2(\alpha), \dots, f_r(\alpha) \} \in K^M_l(k_v)$ in Lemma \ref{generator-residue-field},
we may express $N_{k_v/k}(x)$ as a sum of $\pm \{-1, -1, \dots, -1\}$ and
$N_{v_i} (x_i)$ in $K^M_l(k)$, for $i=1,2,\dots,r$, where $\deg \pi_{v_i} < \deg \pi_v$ and $x_i \in K^M_l(k_{v_i})$ for each $i$.
\end{lemma}
\begin{proof}
Let $d$ be the degree of the monic irreducible polynomial $\pi_v$ in $k[X]$.
By Lemma \ref{generator-residue-field}, $K^M_l(k_v)$ is generated by symbols of the form
$$x = \{a_1, a_2, \dots, a_{l-r}, f_1(\alpha), f_2(\alpha), \dots, f_r(\alpha) \},$$
where $a_1, a_2, \dots, a_{l-r}$ are elements of $k^\times$ and $f_1, \dots, f_r$ are monic irreducible polynomials in $k[X]$
with $0 < \deg f_1 < \deg f_2 < \dots < \deg f_r < d$.
Let $v_1, v_2, \dots, v_r$ be the discrete valuations of $K=k(X)$ associated with $f_1, f_2, \dots, f_r$, respectively.
Write $y= \{a_1, a_2, \dots, a_{l-r}, f_1, f_2, \dots, f_r, \pi_v \} \in K^M_{l+1}(K)$ so that $\partial_v(y)=x$.
Then $N_{k_v/k}(x)$ appears as a term in the Weil reciprocity law $\displaystyle \sum_{w} N_w \left( \partial_w (x) \right)=0$.
But, we have $\partial_w(y) = 0$ unless $w$ is equal to either $v$, $v_i$ for some $i \in \{1,2,\dots,r\}$ or $v_\infty$.
Note that $\partial_{v_i}(y) = (-1)^{r-i}x_i$, where
$x_i = \{a_1, a_2, \dots, a_{l-r}, f_1(\alpha_i), \dots, f_{i-1}(\alpha_i), f_{i+1}(\alpha_i), \dots, f_r(\alpha_i), \pi_v(\alpha_i) \}$
and $\alpha_i$ is the image of $X$ in $k_v$ under the identification $k_v = k[X]/(f_i)$.
Also, we have $\partial_{v_\infty}(y) = (-1)^{r+1} \deg(f_1) \dots \deg(f_r) \deg(\pi_v) \{-1, -1, \dots, -1\}$.
Therefore,
$$N_{k_v/k}(x) = (-1)^{r} \deg(f_1) \dots \deg(f_r) \deg(\pi_v) \{-1, \dots, -1\} - \sum_{i=1}^r (-1)^r N_{v_i}(x_i).$$
Note that each $x_i$ may be written explicitly once $x$ is known.
\end{proof}

\section{The Goodwillie groups} \label{sec-Goodwillie}

We define the $l$-th Goodwillie group $GW_l(k)$ for $l \ge 1$ as follows when $k$ is a field.

\begin{definition} \label{GW_l}
$GW_l(k)$ is the abelian group generated by $l$-tuples of commuting matrices $(A_1, \dots, A_l)$ ($A_1, \dots, A_l \in GL_n(k)$ for various $n \ge 1$),
subject to the following 4 kinds of relations. \\
(i) (Identity Matrices) $(A_1, \dots, A_l) = 0$ when $A_i$ for some $i$ is equal to the identity matrix $I_n \in GL_n(k)$. \\
(ii) (Similar Matrices) $(A_1, \dots, A_l) = (SA_1S^{-1}, \dots, SA_lS^{-1})$ for commuting $A_1, \dots, A_l \in GL_n(k)$ and any $S \in GL_n(k)$. \\
(iii)  (Direct Sum) $\displaystyle (A_1, \dots, A_l) + (B_1, \dots, B_l)
= \left( \begin{pmatrix} A_1 & 0 \\ 0 & B_1 \end{pmatrix}, \dots, \begin{pmatrix} A_l & 0 \\ 0 & B_l \end{pmatrix} \right)$
for commuting $A_1, \dots, A_l \in GL_n(k)$ and commuting $B_1, \dots, B_l \in GL_m(k)$. \\
(iv) (Polynomial Homotopy) $(A_1(0), \dots, A_l(0)) = (A_1(1), \dots, A_l(1))$ for commuting matrices $A_1(t), \dots, A_l(t)$ in $GL_n(k[t])$, where $k[t]$ is the polynomial ring over $k$ with the indeterminate $t$. \\
\end{definition}

The motivic cohomology $H^d_{\M} \bigl(\Spec \, k , \Z(l) \bigr)$, which appear in various literatures and defined in many different ways although most of them turned out to be isomorphic,
is in fact isomorphic to the Goodwillie group $GW_l(k)$ when the degree $d$ is equal to the weight $l$ (See \cite{WM}).

It is immediate from the definition that, for an arbitrary field extension $k \subset L$, we have a natural homomorphism
$i_{L/k}: GW_l(k) \rightarrow GW_l(L)$ of groups.

Let us denote by $\displaystyle GL(k) = \bigcup_{n \ge 1} GL_n(k)$ the set of invertible matrices of finite ranks over $k$, where
two matrices $A$ and $B$ are considered equal if $A = \begin{pmatrix} B & 0 \\ 0 & I \\ \end{pmatrix}$ for some identity matrix $I$, or the other
way around. Then $GW_1(k)$ can be thought of as generated by elements of $GL(k)$, because of the relations $(i)$ and $(iii)$.

The following lemma is easily proved using the definition of $GW_1(k)$.

\begin{lemma} \label{basic-GW_1}
(i) Every elementary matrix represents 0 in $GW_1(k)$. \\
(ii) $(A) + (B) = (AB)$ for $A, B \in GL_n(k)$ in $GW_1(k)$
\end{lemma}
\begin{proof}
$(i)$ An immediate consequence of the relation $(iv)$ is that the element of $GW_1(k)$ represented by
a $2\times 2$ block matrix $\displaystyle \begin{pmatrix} A & C \\ 0 & B  \end{pmatrix}$ is equal to the one represented by
$\displaystyle \begin{pmatrix} A & 0 \\ 0 & B  \end{pmatrix}$. To see this, simply take the invertible polynomial matrix
$\displaystyle \begin{pmatrix} A & Ct \\ 0 & B  \end{pmatrix}$ which have entries in $k[t]$ and put $t=0$ and $t=1$.
By the same reason, we see that an elementary matrix $E_{ij} \in GL_n(k)$, whose diagonal entries and $(i,j)$-th term are 1 and other entries are all 0,
represents 0 in $GW_1(k)$.

$(ii)$ In $GW_1(k)$, we have
$$\begin{pmatrix} B & 0 \\ 0 & A \end{pmatrix} = \begin{pmatrix} B & I \\ 0 & A \end{pmatrix}
= \begin{pmatrix} I & 0 \\ B & I \end{pmatrix} \begin{pmatrix} B & I \\ 0 & A \end{pmatrix} \begin{pmatrix} I & 0 \\ -B & I \end{pmatrix}
= \begin{pmatrix} 0 & I \\ -AB & A+B \end{pmatrix}$$
On the other hand, we also have
$$(AB) = \begin{pmatrix} I & 0 \\ 0 & AB \end{pmatrix}
=\begin{pmatrix} I & I \\ 0 & AB \end{pmatrix}
= \begin{pmatrix} I & 0 \\ I & I \end{pmatrix} \begin{pmatrix} I & I \\ 0 & AB \end{pmatrix} \begin{pmatrix} I & 0 \\ -I & I \end{pmatrix}
= \begin{pmatrix} 0 & I \\ -AB & I+AB \end{pmatrix}.$$
But, by letting $t=0$ and $t=1$ in the polynomial homotopy
$$ \begin{pmatrix} 0 & I \\ -AB & t(A+B) + (1-t)(I+AB) \end{pmatrix} \ \in \ GL_{2n}(k[t]),$$ we see that
$$ \begin{pmatrix} 0 & I \\ -AB & A+B \end{pmatrix} = \begin{pmatrix} 0 & I \\ -AB & I+AB \end{pmatrix} \ {\rm in}\ GW_1(k)$$
and we are done.
\end{proof}

\begin{corollary} \label{single-rep-GW_1}
Every element of $GW_1(k)$ can be written as $(A)$ for some single invertible matrix $A \in GL(k)$.
\end{corollary}

\begin{proof}
This follows from the simple observation that $\displaystyle (A) - (B) = (A) +  (B^{-1}) = \begin{pmatrix} A & 0 \\ 0 & B^{-1} \end{pmatrix}$ in $GW_1(k)$.
\end{proof}

\begin{prop} \label{GW1-field}
We have $GW_1(k) \simeq k^\times$, the multiplicative group of units in $k$.
\end{prop}
\begin{proof}
Define the map $\phi: GW_1(k) \rightarrow k^\times$ by $\displaystyle \phi \big(\sum_{i} n_i (A_i)\big) = \prod_{i} (\det A_i )^{n_i}$.
Then, $\phi$ is easily checked to be well-defined using Definition \ref{GW_l} and is clearly surjective.

On the other hand, let $(A)$, where $A$ is an invertible matrix in $GL(k)$, be an arbitrary element in the kernel of $\phi$ (c.f. Corollary \ref{single-rep-GW_1}).
When the determinant of $A$ is 1,
it is well-known (c.f. 2.2.6 \& 2.3.2 in \cite{MR1282290}) that $A$ is a product of elementary matrices, which are trivial in
$GW_1(k)$ by Lemma \ref{basic-GW_1} $(i)$.
\end{proof}

To understand the Goodwillie group $GW_l(k)$ as a homology group of a complex, we introduce the following notation.
\begin{definition} \label{GW-complex}
$C(k[t_1, \dots, k_d], l)$($d \ge 0, l \ge 1$) is defined to be the abelian group
generated by $l$-tuples $\left( A_1,\dots,A_l \right) = \left( A_1(t_1, \dots, t_d),\dots,A_l(t_1, \dots, t_d)\right)$
where $A_1=A_1(t_1, \dots, t_d),\dots,A_l=A_l(t_1, \dots, t_d)$ are commuting matrices in
$GL_n(k [t_1, \dots, t_d])$ for various $n \ge 1$, subject to the following 3 kinds of relations. \\
(i) (Identity Matrices) $\left( A_1,\dots,A_l \right) = 0$ when $A_i$, for some $i$, is equal to the identity matrix $I_n$ of rank $n$. \\
(ii) (Similar Matrices) $\left( A_1,\dots,A_l \right) = (S A_1 S^{-1}, \dots, S A_l S^{-1})$
                        for any $S \in GL_n(k [t_1, \dots, t_d])$. \\
(iii)  (Direct Sum) $\displaystyle \left( A_1,\dots,A_l \right) + \left( B_1,\dots,B_l \right)
= \left( \begin{pmatrix} A_1 & 0 \\ 0 & B_1 \end{pmatrix}, \dots, \begin{pmatrix} A_l & 0 \\ 0 & B_l \end{pmatrix} \right)$
for commuting $A_1,\dots,A_l \in GL_n(k [t_1, \dots, t_d])$ and commuting $B_1, \dots, B_l \in GL_m(k [t_1, \dots, t_d])$.
\end{definition}

Our main interest in this article is when $d=0$ and $d=1$.
When $d=1$, we set $t=t_1$ and we define the boundary map $\partial: C(k[t], l) \rightarrow C(k, l)$ by sending
$\left( A_1(t),\dots,A_l(t) \right)$ in $C(k[t], l)$ to $\left( A_1(1),\dots,A_l(1) \right) - \left( A_1(0),\dots,A_l(0) \right)$ in $C(k, l)$.
Then, the Goodwillie group $GW_l(k)$ is nothing but the cokernel of the map $\partial: C(k[t], l) \rightarrow C(k, l)$.
We will denote by the same notation $(A_1,\dots,A_l)$ the element in $C(k, l) / \partial C(k[t], l) = GW_l(k)$
represented by $(A_1,\dots,A_l)$, by abuse of notation, whenever $A_1,\dots,A_l$ are commuting matrices in $GL_n(k)$.

Before we proceed to the next section, we define products in the Goodwillie groups.
Recall that, for $A=(a_{ij}) \in GL_m(k)$ and $B=(b_{ij}) \in GL_n(k)$, the Kronecker product $A \otimes B$ is defined to be the block matrix in $GL_{mn}(k)$
whose $(i,j)$-th block ($1 \le i,j \le m$) is given by $a_{ij} B$, i.e.,
$$A \otimes B = \begin{pmatrix} a_{11} B & \dots & a_{1m}B \\
                            \vdots & \ddots & \vdots \\
                            a_{m1}B & \dots & a_{mm} B
\end{pmatrix}.$$
One of the basic properties of the Kronecker product of matrices is that
$(A \otimes B) (C \otimes D) = AC \otimes BD$ when $A, C \in GL_m(k)$ and $B, D \in GL_n(k)$
and thus $(A \otimes B)^{-1} = A^{-1} \otimes B^{-1}$.

\begin{definition} \label{product-GW}
A product $\cdot : GW_p(k) \times GW_q(k) \rightarrow GW_{p+q}(k)$ for $p,q \ge 1$ is defined as follows.
For commuting matrices $A_1, \dots, A_p \in GL_n(k)$ and commuting matrices $B_1, \dots, B_q \in GL_m(k)$,
$(A_1, \dots, A_p) \cdot (B_1, \dots, B_q) \in GW_{p+q}(k)$ is represented by the symbol
$(A_1 \otimes I_n, \dots, A_p \otimes I_n, I_m \otimes B_1, \dots, I_m \otimes B_q )$ where
$I_m$ and $I_n$ are the identity matrices of rank $m$ and $n$, respectively.
\end{definition}

Let us verify that the product $\cdot : GW_p(k) \times GW_q(k) \rightarrow GW_{p+q}(k)$ is well-defined.
First of all, $A_i \otimes I_n$ and $I_m \otimes B_j$ ($1 \le i \le p$, $1 \le j \le q$) are easily seen to be commuting.
If either the first factor from $GW_p(k)$ or the second factor from $GW_q(k)$ is in one of the relations in Definition \ref{GW_l}, let us check that
their product as in Definition \ref{product-GW} is also in the relations.
We will check this fact only when the first factor from $GW_p(k)$ is in the relations.\\
$(i)$ If, say, $A_i$ is the identity matrix, then $A_i \otimes I_n$ is the identity matrix. \\
$(ii)$ If $S$ is in $GL_m(k)$, then we have, for $T =S \otimes I_n$,
\begin{multline*} \left((S^{-1}A_1 S) \otimes I_n, \dots, (S^{-1} A_p S ) \otimes I_n, I_m \otimes B_1, \dots, I_m \otimes B_q \right) \\
=\left( T^{-1} (A_1 \otimes I_n) T \otimes I_n, \dots, T^{-1} (A_p \otimes I_n) T, T^{-1} (I_m \otimes B_1) T, \dots, T^{-1}(I_m \otimes B_q)T \right).
\end{multline*}
$(iii)$ For commuting $A_1, \dots, A_p \in GL_m(k)$, commuting $C_1, \dots, C_p \in GL_r(k)$, and commuting $B_1, \dots, B_q \in GL_n(k)$, we have
\begin{multline*}
\left(\begin{pmatrix} A_1 & 0\\ 0 & C_1 \end{pmatrix} \otimes I_n, \dots, \begin{pmatrix} A_p & 0\\ 0 & C_p \end{pmatrix} \otimes I_n, \
I_{m+r} \otimes B_1, \dots, I_{m+r} \otimes B_q \right) \\
= \Biggl(\begin{pmatrix} A_1 \otimes I_n & 0 \\ 0 & C_1 \otimes I_n \end{pmatrix}, \dots, \begin{pmatrix} A_p \otimes I_n & 0 \\ 0 & C_p \otimes I_n \end{pmatrix},
  \begin{pmatrix} I_m \otimes B_1 & 0 \\ 0 & I_r \otimes B_1 \end{pmatrix}, \dots, \begin{pmatrix} I_m \otimes B_q & 0 \\ 0 & I_r \otimes B_q \end{pmatrix} \Biggr).
\end{multline*}
$(iv)$ If $A_1(t), \dots, A_p(t)$ are commuting matrices in $GL_m(k[t])$, then
$A_1(t) \otimes I_n, \dots, A_p(t) \otimes I_n, I_m \otimes B_1, \dots, I_m \otimes B_q$ are commuting matrices in $GL_{mn}(k[t])$ and it is
equal to $(A_1(0) \otimes I_n, \dots, A_p(0) \otimes I_n, I_m \otimes B_1, \dots, I_m \otimes B_q )$ and
$(A_1(1) \otimes I_n, \dots, A_p(1) \otimes I_n, I_m \otimes B_1, \dots, I_m \otimes B_q )$ when $t=0$ and $t=1$, respectively.

We remark that, for commuting matrices $A, B \in GL_m(k)$, it is not necessarily true that $(A) \cdot (B) = (A, B)$ in $GW_2(k)$ unless $A$ and $B$
are $1 \times 1$ matrices.

For the next lemma, we set $GW_0(k) = \Z$, the ring of integers. Then the products $\cdot : GW_0 (k) \times GW_l(k) \rightarrow GW_l(k)$ and
$\cdot : GW_l (k) \times GW_0(k) \rightarrow GW_l(k)$, for $l \ge 0$ can be naturally defined
by considering each $GW_l(k)$ as a $\Z$-module which arises from its abelian group structure.

\begin{lemma} \label{GW-ring}
The product $\cdot : GW_p(k) \times GW_q(k) \rightarrow GW_{p+q}(k)$ makes $\displaystyle \bigoplus_{l \ge 0} GW_l(k)$ into a graded ring.
\end{lemma}
\begin{proof}
To show that it is a graded ring, we need to check that the product is associative and distributive with respect to the addition.
Associativity is easily verified using the property $(A \otimes B) \otimes C = A \otimes (B \otimes C)$ of the Kronecker product of matrices.
We already have the distributive law during the construction.
\end{proof}

  \section{Some properties of Goodwillie groups} \label{sec-prop-Goodwillie}
The basic reference for this section is \cite{MSMulti}.

\begin{lemma} \label{boundary-Z}
The following equalities hold in $GW_l(k)$.

(i) $(B C,A_2,\dots,A_l)=(B,A_2,\dots,A_l)+(C,A_2,\dots,A_l)$,
for commuting matrices $B, C, A_2,\dots,A_l \in GL_n(k)$;

 Also, $(A_1, \dots, A_{i-1}, B C, A_{i+1}, \dots, A_l)
 =(A_1, \dots, A_{i-1}, B, A_{i+1}, \dots, A_l)
 +(A_1, \dots, A_{i-1}, C, A_{i+1}, \dots, A_l)$
for commuting matrices $B, C, A_1,\dots,A_{i-1}, A_{i+1},\dots, A_l$ in $GL_n(k)$;

(ii) $(A_1,\dots,A_i, \dots, A_j, \dots, A_l) = - (A_1,\dots,A_j, \dots, A_i, \dots, A_l)$,
for commuting matrices $A_1,\dots,A_l \in GL_n(k)$;

(iii) $(A_1,\dots,A_i, \dots, A_j, \dots, A_l)=0$, whenever $A_i= -A_j$
for commuting $A_1,\dots,A_l \in GL_n(k)$;

(iv) $(c_1, \dots, \alpha,\dots, 1-\alpha,\dots, c_l)=0$ in $GW_l(k)$, for $\alpha \in k-\{0,1\}$ and $c_i \in k^\times$
for each appropriate $i$.
\end{lemma}

\begin{proof}

$(i)$ Let $H(t)$ be the $2n \times 2n$ matrix
$$\begin{pmatrix} 0 & I_n \\
              -B C & t(I_n+B C)+(1-t)(B+C) \end{pmatrix}. $$
Then, $H(t)$ is in $GL_{2n}(k[t])$, and the image of $\bigl( H(t),\, A_2 \oplus A_2, \dots, A_l \oplus A_l \bigr)$
under $\partial: C(k[t], l) \rightarrow C(k, l)$ is
$(I_n \oplus B C, A_2 \oplus A_2, \dots, A_l \oplus A_l)
-(B \oplus C,A_2 \oplus A_2, \dots, A_l \oplus A_l)
=(B C,A_2, \dots, A_l)
-(B,A_2, \dots, A_l)-(C,A_2, \dots, A_l)$.

The proof is similar for the other cases.

$(ii)$ We let $H(t)$ be the matrix
$$\begin{pmatrix} 0 & I_n \\
              -AB & t(I_n+AB)+(1-t)(A+B) \end{pmatrix}. $$

Then the image of $\bigl( A_1 \oplus A_1,\dots,H(t), \dots, H(t),\dots, A_l \oplus A_l \bigr)$ under $\partial$ is
\begin{align*}
&(A_1,\dots,A_i A_j, \dots, A_i A_j,\dots, \dots, A_l)
-(A_1,\dots,A_i, \dots, A_i,\dots, A_l)
-(A_1,\dots,A_j, \dots, A_j,\dots,  A_l)\\
&= \bigl( (A_1,\dots,A_i, \dots, A_i,\dots, A_l)
+(A_1,\dots,A_i, \dots, A_j,\dots, A_l) \\
& \quad +(A_1,\dots,A_j, \dots, A_i,\dots, A_l)
+(A_1,\dots,A_j, \dots, A_j,\dots, A_l) \bigr) \\
& \quad \quad -(A_1,\dots,A_i, \dots, A_i,\dots,  A_l)
-(A_1,\dots,A_j, \dots, A_j,\dots, A_l) \\
&=(A_1,\dots,A_j, \dots, A_i,\dots, A_l) + (A_1,\dots,A_j, \dots, A_i,\dots, A_l)
\quad \text {in } GW_l(k) \text{ by} \ (i).
\end{align*}

$(iii)$ Let $A=-A_i=A_j$.
It suffices to show that $(-A, A) = 0$ in $GW_2(k)$ as the other cases with $l \ge 3$ follow from this fact
and $(ii)$ using the product structure in Definition \ref{product-GW}.
The image of
$$ \left( \begin{pmatrix} -A & 0 \\ 0 & -A \end{pmatrix}, \begin{pmatrix} 0 & I_n \\ -A & t(A+I_n) \end{pmatrix} \right)$$
under $\partial$ is equal to
{\allowdisplaybreaks \begin{align*}
& \left( \begin{pmatrix} -A & 0 \\ 0 & -A \end{pmatrix}, \begin{pmatrix} 0 & I_n \\ -A & A+I_n \end{pmatrix} \right)
 -\left( \begin{pmatrix} -A & 0 \\ 0 & -A \end{pmatrix}, \begin{pmatrix} 0 & I_n \\ -A & 0 \end{pmatrix} \right) \\
& =\left( \begin{pmatrix} -A & 0 \\ 0 & -A \end{pmatrix},  \begin{pmatrix} A & I_n \\ 0 & I_n \end{pmatrix} \right)
 -\left( \begin{pmatrix} -A & 0 \\ 0 & -A \end{pmatrix}, \begin{pmatrix} 0 & I_n \\ -A & 0 \end{pmatrix} \right) \\
&=(-A, A)
-\left(\begin{pmatrix} -A & 0 \\ 0 & -A \end{pmatrix}, \begin{pmatrix} 0 & I_n \\ -A & 0 \end{pmatrix} \right).
\end{align*}}
So it is enough to show that the element
$$ \left(\begin{pmatrix} -A & 0 \\ 0 & -A \end{pmatrix}, \begin{pmatrix} 0 & I_n \\ -A & 0 \end{pmatrix} \right)$$
vanishes in $GW_l(k)$. But it is equal to
\begin{align*}
\left({\begin{pmatrix} 0 & I_n \\ -A & 0 \end{pmatrix}}^2, \begin{pmatrix} 0 & I_n \\ -A & 0 \end{pmatrix} \right)
= 2 \left({\begin{pmatrix} 0 & I_n \\ -A & 0 \end{pmatrix}}, \begin{pmatrix} 0 & I_n \\ -A & 0 \end{pmatrix} \right),
\end{align*}
which vanishes in $GW_l(k)$ by $(ii)$ above.

$(iv)$ First, we show that, in $GW_l(k)$,
\begin{align} \label{inv-a.1-a}
(c_1, \dots, b,\dots, 1-b, \dots, c_l)=(c_1, \dots, a,\dots, 1-a,\dots, c_l),
\end{align}
for $a, b \in k-\{0,1\}$ and $c_i \in k^\times$ for each appropriate $i$.
To show (\ref{inv-a.1-a}), it suffices to show that $(a, 1-a) = (b, 1-b)$ in $GW_2(k)$ as the other cases follow from this fact by
multiplying by $(c_1, \dots, c_l)$ (with an appropriate omission of indices) using the product structure in Definition \ref{product-GW}.
Let us take
$$A(t) = \begin{pmatrix} 0 & 0 & -ab \\
                        1 & 0 & at + b(1-t) \\
                        0 & 1 & a(1-t)+bt \end{pmatrix}$$
to be the companion matrix of the monic polynomial $X^3 + (a(t-1)-bt)X^2 + (b(t-1)-at)X+ab$ with coefficients
in $k[t]$. Then, both $A(t)$ and $I_3-A(t)$ are in $GL_3(k[t])$ with determinant $-ab$ and $(1-a)(1-b)$, respectively.
Note also that the eigenvalues of $A(0)$ and $A(1)$ are $a, \sqrt{b},-\sqrt{b}$ and $-\sqrt{a},\sqrt{a}, b$, respectively, in some suitable algebraic extension of $k$.
Take $z=2 \bigl( A(t), I_3-A(t) \bigr)$. By the theory of rational canonical form, $\partial z$ is equal to
{\allowdisplaybreaks
\begin{align*}
& 2 \left( (b ,1-b)
+\left( \begin{pmatrix} 0 & 1 \\ a & 0 \end{pmatrix}, \begin{pmatrix} 1 & -1 \\ -a & 1 \end{pmatrix} \right) \right)
 -2 \left( (a, 1-a)
+\left( \begin{pmatrix} 0 & 1 \\ b & 0 \end{pmatrix}, \begin{pmatrix} 1 & -1 \\ -b & 1 \end{pmatrix} \right) \right) \\
&= -2 (a, 1-a) + 2(b,1-b)
 -\left( {\begin{pmatrix} 0 & 1 \\ b & 0 \end{pmatrix}}^2,  \begin{pmatrix} 1 & -1 \\ -b & 1 \end{pmatrix} \right)
 +\left({\begin{pmatrix} 0 & 1 \\ a & 0 \end{pmatrix}}^2, \begin{pmatrix} 1 & -1 \\ -a & 1 \end{pmatrix} \right)\\
&= \left(\begin{pmatrix} b & 0 \\ 0 & b \end{pmatrix},  \begin{pmatrix} 1-b & 0 \\ 0 & 1-b \end{pmatrix} \right)
 - \left(\begin{pmatrix} b & 0 \\ 0 & b \end{pmatrix}, \begin{pmatrix} 1 & -1 \\ -b & 1 \end{pmatrix} \right) \\
&\quad \quad -\left( \begin{pmatrix} a & 0 \\ 0 & a \end{pmatrix}, \begin{pmatrix} 1-a & 0 \\ 0 & 1-a \end{pmatrix} \right)
+ \left( \begin{pmatrix} a & 0 \\ 0 & a \end{pmatrix}, \begin{pmatrix} 1 & -1 \\ -a & 1 \end{pmatrix} \right) \\
&=\left(\begin{pmatrix} b & 0 \\ 0 & b \end{pmatrix},
\begin{pmatrix} 1-b & 0 \\ 0 & 1-b \end{pmatrix} {\begin{pmatrix} 1 & -1 \\ -b & 1 \end{pmatrix}}^{-1} \right)
 -\left( \begin{pmatrix} a & 0 \\ 0 & a \end{pmatrix},
 \begin{pmatrix} 1-a & 0 \\ 0 & 1-a \end{pmatrix} {\begin{pmatrix} 1 & -1 \\ -a & 1 \end{pmatrix}}^{-1}  \right)\\
&=\left( \begin{pmatrix} b & 0 \\ 0 & b \end{pmatrix}, \begin{pmatrix} 1 & 1 \\ b & 1 \end{pmatrix} \right)
-\left(\begin{pmatrix} a & 0 \\ 0 & a \end{pmatrix},  \begin{pmatrix} 1 & 1 \\ a & 1 \end{pmatrix} \right) \\
&=\left(\begin{pmatrix} b & 0 \\ 0 & b \end{pmatrix}, \begin{pmatrix} {\frac {-b} {1-b}} & {\frac 1 {1-b}} \\ 0 & 1 \end{pmatrix}
            \begin{pmatrix} 1 & 1 \\ b & 1 \end{pmatrix} {\begin{pmatrix} {\frac {-b}{1-b}} & {\frac 1 {1-b}} \\ 0 & 1 \end{pmatrix}}^{-1}  \right)
-\left(\begin{pmatrix} a & 0 \\ 0 & a \end{pmatrix}, \begin{pmatrix} {\frac {-a} {1-a}} & {\frac 1 {1-a}} \\ 0 & 1 \end{pmatrix}
           \begin{pmatrix} 1 & 1 \\ a & 1 \end{pmatrix} {\begin{pmatrix} {\frac {-a}{1-a}} & {\frac 1 {1-a}} \\ 0 & 1 \end{pmatrix}}^{-1} \right)\\
&=\left(\begin{pmatrix} b & 0 \\ 0 & b \end{pmatrix}, \begin{pmatrix} 0 & 1 \\ b-1 & 2 \end{pmatrix} \right)
-\left( \begin{pmatrix} a & 0 \\ 0 & a \end{pmatrix}, \begin{pmatrix} 0 & 1 \\ a-1 & 2 \end{pmatrix} \right).
\end{align*}}
By taking image under $\partial$ of the element
{\allowdisplaybreaks \begin{align*}
&\left( \begin{pmatrix} b & 0 \\ 0 & b \end{pmatrix}, \begin{pmatrix} 0 & 1 \\ b-1 & (2-b)t+2(1-t) \end{pmatrix} \right)
-\left( \begin{pmatrix} a & 0 \\ 0 & a \end{pmatrix}, \begin{pmatrix} 0 & 1 \\ a-1 & (2-a)t+2(1-t) \end{pmatrix} \right),
\end{align*}}
we see that
{\allowdisplaybreaks \begin{align*}
\partial z
&= \left(\begin{pmatrix} b & 0 \\ 0 & b \end{pmatrix}, \begin{pmatrix} 0 & 1 \\ b-1 & 2-b \end{pmatrix} \right)
-\left(\begin{pmatrix} a & 0 \\ 0 & a \end{pmatrix}, \begin{pmatrix} 0 & 1 \\ a-1 & 2-a \end{pmatrix} \right) \\
&= \left(\begin{pmatrix} b & 0 \\ 0 & b \end{pmatrix}, \begin{pmatrix} 1-b & 0 \\ 0 & 1 \end{pmatrix} \right)
    -\left(\begin{pmatrix} a & 0 \\ 0 & a \end{pmatrix}, \begin{pmatrix} 1-a & 0 \\ 0 & 1 \end{pmatrix} \right)\\
&= ( b, 1-b )-( a, 1-a) \\
\end{align*}}
in $GW_l(k)$. Therefore, the equality (\ref{inv-a.1-a}) holds in $GW_l(k)$.

Now we give a proof of $(iv)$. As in the proof of the equality (\ref{inv-a.1-a}) above, it suffices to prove that $(\alpha, 1-\alpha)=0$
in $GW_2(k)$ when $\alpha \in k^\times - \{1\}$.
When $k$ is the field $\F_2$ with 2 elements, it is a vacuous statement.
If $k$ is the field $\F_3$ with 3 elements, $\alpha=2$ is the only choice and $(2, 1-2) = (2, (-2)^2) = 2 (2,-2) =0$ by $(iii)$.
Therefore, we may assume that $k$ has more than 3 elements and there exists an element $e \in k$ such that $e^3-e \neq 0$.
By the equality (\ref{inv-a.1-a}), with $a=e,\, b=1-e$, we have
$(e, 1-e )-( 1-e, e ) = 2 (e, 1-e) = 0$ in $GW(k)$.
With $a=-e,\, b=1+e$ in (\ref{inv-a.1-a}), we have $2( e, 1+e ) = 2( -e, 1+e )=0$.
Hence, $(e^2, 1-e^2 ) = 2(e, 1-e )+2(e, 1+e ) = 0$.

On the other hand, by the equality (\ref{inv-a.1-a}) with $a=\alpha,\, b=e^2$,
we see that $(\alpha, 1-\alpha )=(e^2, 1-e^2 ) =0$ in $GW(k)$ and we are done.
\end{proof}

\begin{corollary} \label{multilin-l-l} (Multilinearity and Skew-symmetry for $GW_l(k)$)

(i) $\displaystyle (A_1, \dots, A_{i-1}, B C, A_{i+1}, \dots, A_l) =
 (A_1, \dots, A_{i-1}, B, A_{i+1}, \dots, A_l) +(A_1, \dots, A_{i-1}, C, A_{i+1}, \dots, A_l)$
 in $GW_l(k)$, for arbitrary commuting matrices $B, C, A_1,\dots,A_{i-1}, A_{i+1},\dots, A_l$ in $GL_n(k)$,

(ii) $(A_1,\dots,A_i, \dots, A_j, \dots, A_l) = - (A_1,\dots,A_j, \dots, A_i, \dots, A_l)$
in $GW_l(k)$
for arbitrary commuting matrices $A_1,\dots,A_l$ in $GL_n(k)$.
\end{corollary}

If $A_1,\dots,A_l$ and $A'_1,\dots,A'_l$ are $l$-tuples of commuting matrices in $GL_n(k)$ and $GL_m(k)$, respectively, then
$\displaystyle (A_1,\dots,A_l) + (A'_1,\dots,A'_l) = (A_1 \oplus A'_1,\dots,A_l \oplus A'_l)$
in $GW_l(k)$. Therefore, we obtain the following result from Corollary \ref{multilin-l-l}, which may be viewed as a generalization of
Corollary \ref{single-rep-GW_1}.
\begin{corollary}
Every element in $GW_l(k)$ can be written as a single symbol $(A_1,\dots,A_l)$, where
$A_1,\dots,A_l$ are commuting matrices in some $GL_n(k)$.
\end{corollary}

\begin{corollary} \label{GW-commutative}
The graded ring $\displaystyle \bigoplus_{l \ge 0} GW_l(k)$ in Lemma \ref{GW-ring} is anti-commutative.
\end{corollary}
\begin{proof}
Note that, when $A \in GL_m(k)$ and $B \in GL_n(k)$, then
the matrix $A \otimes B$ is similar to $B \otimes A$
by the similarity matrix $S$ which is given by a base change
which sends the $(i,j)$-th canonical vector to $(j,i)$-th canonical vector for each $i=1,\dots,m$ and $j=1,\dots,n$.
Therefore, for commuting $A_1, \dots, A_p \in GL_m(k)$ and commuting $B_1, \dots, B_q \in GL_n(k)$, we have
$\displaystyle (A_1, \dots, A_p) \cdot (B_1, \dots, B_q) = (A_1 \otimes I_n, \dots, A_p \otimes I_n, I_m \otimes B_1, \dots, I_m \otimes B_q )
=(I_n \otimes A_1, \dots, I_n \otimes A_p, B_1 \otimes I_m, \dots, B_q \otimes I_m)$ by Definition \ref{GW_l} $(ii)$.
But, by Corollary \ref{multilin-l-l}, it is equal to $(-1)^{pq} (B_1 \otimes I_m, \dots, B_q \otimes I_m, I_n \otimes A_1, \dots, I_n \otimes A_p)
=(-1)^{pq} (B_1, \dots, B_q) \cdot (A_1, \dots, A_p).$
\end{proof}

In fact, this corollary can be also deduced from Lemma \ref{Milnor-basic-relation} $(ii)$ and Theorem \ref{Milnor-iso}, which is proved later.
Thanks to Lemma \ref{boundary-Z}, we may now construct a map from the Milnor's $K$-groups to the Goodwillie groups.

\begin{prop} \label{Milnor-map}
The assignment $\{ a_1,a_2, \dots, a_l \} \mapsto (a_1,a_2,\dots,a_l) $ for each Milnor symbol
$\{a_1,a_2, \dots, a_l \}$ gives rise to a homomorphism $\rho_l$
from the Milnor's $K$-group $K^M_l(k)$ to $GW_l(k)$ for any field $k$. In fact, it gives a graded ring homomorphism
from $\displaystyle \bigoplus_{l \ge 0} K^M_l(k)$ to $\displaystyle \bigoplus_{l \ge 0} GW_l(k)$.
\end{prop}
\begin{proof}
The case $l=1$ is treated in Proposition \ref{GW1-field}.
So we may suppose that $l \ge 2$.
Thanks to Corollary \ref{multilin-l-l} $(i)$, it suffices to
show that for every $\alpha \in k-\{0,1\}$ and $c_r \in k^\times$ for $1 \le r \le l$, $r \ne i,j$,
$(c_1, \dots, \alpha, \dots, 1-\alpha, \dots, c_l)$ is in $\partial C(k[t], l)$.
But, this is already done in Lemma \ref{boundary-Z} $(iv)$.
The fact that the $\rho_l$ gives rise to a graded ring homomorphism follows from the definition of products.
\end{proof}

\section{Transfer maps for Goodwillie groups and relations with Milnor's $K$-groups}
\label{sec-transfer-Goodwillie}

In case of the Goodwillie groups, there is an immediate functorial definition of the transfer maps
for any finite field extension $k \subset L$.

\begin{definition} If $A_1, \dots, A_l$ are commuting invertible matrices of rank $n$ with coefficients
in $L[t]$ (respectively, $L$),
then by identifying $L[t]$ (respectively $L$) as a free $k[t]$-module (respectively $k$-module) of rank $d=[L:k]$,
we may consider $A_1, \dots, A_l$ as commuting invertible linear maps on $L[t]^n$ (respectively $L^n$), i.e., commuting invertible linear maps
on $k[t]^{nd}$ (respectively, $k^{nd}$).
So, they are associated with certain commuting invertible matrices $A'_1, \dots, A'_l$ of rank $nd$ with coefficients in $k[t]$ (respectively, $k$).
This gives a map $C(L[t], l) \rightarrow C(k[t], l)$ (see Definition \ref{GW-complex})
(respectively, a map $C(L, l) \rightarrow C(k, l)$). These maps are compatible with
the boundary map $\partial$ and thus they induce a homomorphism $N_{L/k}: \ GW_l(L) \rightarrow GW_l(k)$, which is called the
transfer map for the Goodwillie group.
\end{definition}

For convenience, we define $N_{L/k}: GW_0(L)=\Z \rightarrow GW_0(k)=\Z$ to be a multiplication by the degree $[L:k]$ of the field extension.
It is immediate from the definition that $N_{L'/L} \circ N_{L/k} = N_{L'/k}$ whenever we have a tower of finite field extensions $k \subset L \subset L'$.

If $d=[L:k]$, $L$ is isomorphic to $k^{\oplus d}$ as $k$-vector space.
So, a multiplication by a matrix $A$ of rank $n$ with entries in $k$ induces a $k$-linear map on $L$,
which is associated with the matrix $A \otimes I_d$ of rank $nd$ whose diagonal blocks are equal to $A$.
Therefore, the composition
$$ \xymatrix {GW_l(k) \ar[r]^-{i_{L/k}} & GW_l(L) \ar[r]^-{N_{L/k}} & GW_l(k) },$$
where the first map $i_{L/k}$ is induced by the inclusion of the fields $k \subset L$,
is just a multiplication by $d$.

More generally, we have the following projection formula.

\begin{lemma} (Projection formula) \label{projection-formula}
For $z \in GW_p(k)$ and $w \in GW_q(L)$, we have $N_{L/k}( i_{L/k}(z) \cdot w ) = z \cdot N_{L/k}(w)$ in $GW_{p+q}(k)$, where the products
$\cdot$ are defined as in Definition \ref{product-GW}.
\end{lemma}
\begin{proof}
We may assume that $z = (A_1, \dots, A_p)$ and $w = (B_1, \dots, B_q)$
for some commuting $A_1, \dots, A_p \in GL_m(k)$ and commuting $B_1, \dots, B_q \in GL_n(L)$.

For a fixed basis of $L$ as a $k$-vector space, let $B'_1, \dots, B'_q$ be the commuting matrices in $GL_{nd}(k)$ which are associated
with the commuting invertible linear maps induced by $B_1, \dots, B_q$ on $L^n \simeq k^{nd}$.
Then we have
\begin{align*}
N_{L/k} \left( (A_1, \dots, A_p) \cdot (B_1, \dots, B_q) \right)
&=N_{L/k}\left( A_1 \otimes I_n, \dots, A_p \otimes I_n, I_m \otimes B_1, \dots, I_m \otimes B_q \right) \\
&= \left( (A_1 \otimes I_n) \otimes I_d, \dots, (A_p \otimes I_{n})\otimes I_d , I_m \otimes B'_1, \dots, I_m \otimes B'_q \right) \\
&=  (A_1, \dots, A_p) \cdot N_{L/k} \left( (B_1, \dots, B_q) \right).
\end{align*}
The case when $p=0$ (or $q=0$) is already shown above and the proof is complete.
\end{proof}

Note that we also have $N_{L/k}( w \cdot i_{L/k}(z) ) =  N_{L/k}(w) \cdot z$ by Corollary \ref{GW-commutative}.
By Proposition \ref{GW1-field}, every element in $GW_l(k)$ can be identified with a symbol represented by its determinant,
so we have the following corollary from the projection formula.

\begin{corollary} \label{projection-determinant-formula}
Suppose that $\alpha_1, \dots, \alpha_l \in k^\times$ and $\beta \in L^\times$.
Then $$N_{L/k} \left( \alpha_1, \dots, \alpha_l, \beta \right) = \left(\alpha_1, \dots, \alpha_l, N_{L/k}(\beta) \right),$$
where $N_{L/k}(\beta) \in k^\times$ is the image of $\beta$ under the usual norm map $N_{L/k}: L^\times \rightarrow k^\times$.
By Corollary \ref{multilin-l-l} (ii), even if $\beta$ is not located in the last coordinate, a similar equality holds.
\end{corollary}

On the other hand, we have seen that the transfer maps $N_{L/k}: \ K^M_l (L) \rightarrow K^M_l (k)$ for the Milnor's $K$-groups are defined
whenever $L/k$ is a finite field extension in Section \ref{sec-transfer-Milnor}.

We will need the following lemma to prove the compatibility between the two transfer maps.

\begin{lemma} \label{Weil-like-formula}
For monic irreducible polynomials $f_0(X), \dots, f_{l}(X)$ in $k[X]$, which are relatively prime, we have
$\displaystyle \sum_{v} N_{k_v /k} \left( \rho_l \partial_v \{ f_0(X), \dots, f_{l}(X) \} \right) = 0,$
where the sum is taken over all discrete valuations, including $v_\infty$ on $k(X)$, which vanish on $k$.
\end{lemma}

\begin{proof}
Since $f_0(X), \dots, f_{l}(X)$ are monic, we have, in $K^M_l(k)$,
$$ \partial_{v_\infty} \{f_0(X), \dots, f_{l}(X) \} = (-1)^{l+1} \deg f_0(X) \dots \deg f_{l}(X) \{-1,-1, \dots, -1\}.$$
So we need to show that
$$ \sum_{v \ne v_\infty} N_{k_v /k} \left( \rho_l \partial_v \{ f_0(X), \dots, f_{l}(X) \} \right)
= (-1)^l \deg f_0(X) \dots \deg f_{l}(X) \left(-1,-1, \dots, -1 \right).$$

We first illustrate the case $l=2$.
Let $f(X), g(X)$, and $h(X)$ be monic irreducible polynomials in $k[X]$. Then
$\partial_v \{f(X), g(X), h(X) \}$, when $v \ne v_\infty$, vanishes unless $v$ is one of the discrete valuations $v_f$, $v_g$, and $v_h$ associated with
$f(X)$, $g(X)$, and $h(X)$, respectively.

Write $\displaystyle f(X) = \prod_{i=1}^m (X-\alpha_i)$, $\displaystyle g(X) = \prod_{i=1}^n (X-\beta_i)$,
and $\displaystyle h(X) = \prod_{i=1}^q (X-\gamma_i)$ with $\alpha = \alpha_1,
\beta = \beta_1, \gamma=\gamma_1$.
Then $\partial_{v_f} \{f(X), g(X), h(X) \} = \{ \overline{g(X)}, \overline{h(X)} \}$ in $K^M_2 (L)$, where $L=k[X]/(f(X)) \simeq k(\alpha)$.
Hence, $\partial_{v_f} \{f(X), g(X), h(X) \} = \{ g(\alpha), h(\alpha) \}$.

Let us show that the following equality is true in $GW_2(k)$.
\begin{align} \label{norm-equ1}
 N_{k_{v_f} /k} (g(\alpha), h(\alpha))
= {\frac {\deg f(X) \deg g(X) \deg h(X)} {[k(\alpha, \beta, \gamma):k]}} N_{k(\alpha, \beta, \gamma)/k} (\alpha - \beta, \alpha - \gamma)
\end{align}
We first observe that if $K/k$ is a finite field extension and $\phi: K \rightarrow K'$ is a $k$-linear field isomorphism, then
$N_{K/k} (a, b) = N_{K'/k} (\phi(a), \phi(b))$ since it is true when $K$ is a simple extension by definition.
Next, we factor $g(X)=g_1(X) \dots g_s(X)$ and $h(X)= h_1(X) \dots h_t(X)$ where $g_1, \dots, g_s$ are monic irreducible in $k(\alpha)[X]$ and
$h_1, \dots, h_t$ are monic irreducible in $k(\alpha, \beta)[X]$ so that $t = {\deg(g)}/ {[k(\alpha, \beta), k(\alpha)]}$
and $s= {\deg(h)}/ {[k(\alpha, \beta, \gamma), k(\alpha, \beta)]}$.
By rearranging $\beta_i$ and $\gamma_j$ if necessary, let $\beta_1=\beta, \beta_2, \dots, \beta_s$ and $\gamma_1=\gamma, \gamma_2,
\dots, \gamma_t$ be roots of $g_1, g_2, \dots, g_s$ and $h_1, h_2, \dots, h_t$, respectively. These elements reside in an algebraic closure of $k$.
By the above observation, we see that
$N_{k(\alpha, \beta, \gamma)/k} (\alpha - \beta, \alpha - \gamma)= N_{k(\alpha, \beta_i, \gamma_j)/k}(\alpha - \beta_i, \alpha - \gamma_j)$
for each $i=1, \dots, s$ and $j=1, \dots, t$. Hence, by Corollary \ref{projection-determinant-formula} and Corollary \ref{multilin-l-l} $(i)$,
{\allowdisplaybreaks \begin{align*}
st N_{k(\alpha, \beta, \gamma)/k} (\alpha - \beta, \alpha - \gamma)
& = \sum_{i=1, \dots, s}\sum_{j=1, \dots, t} N_{k(\alpha, \beta_i, \gamma_j)/k}(\alpha - \beta_i, \alpha - \gamma_j) \\
& = \sum_{i=1, \dots, s}\sum_{j=1, \dots, t} N_{k(\alpha, \beta_i)/k} \Big( N_{k(\alpha, \beta_i, \gamma_j)/k} \big(\alpha - \beta_i, h_j(\alpha) \big)\Big) \\
& = \sum_{i=1, \dots, s}\sum_{j=1, \dots, t} N_{k(\alpha )/k} \Big( N_{k(\alpha, \beta_i)/k(\alpha)}\big(\alpha - \beta_i, h_j(\alpha)\big) \Big)\\
& = \sum_{i=1, \dots, s}\sum_{j=1, \dots, t} N_{k(\alpha )/k} \big( g_i(\alpha), h_j(\alpha) \big)
=  N_{k(\alpha )/k} \big( g(\alpha), h(\alpha) \big).
\end{align*}}
Since $\displaystyle st = {\frac {\deg(h)} {[k(\alpha, \beta, \gamma), k(\alpha, \beta)]}} {\frac {\deg(g)} {[k(\alpha, \beta), k(\alpha)]}}
= {\frac {\deg f(X) \deg g(X) \deg h(X)} {[k(\alpha, \beta, \gamma):k]}}$, we have verified the equality (\ref{norm-equ1}).
Similarly, we have
\begin{align*} & N_{k_{v_g} /k} \rho_2 \partial_{v_g} \{f(X), g(X), h(X) \} = {\frac {\deg f(X) \deg g(X) \deg h(X)} {[k(\alpha, \beta, \gamma):k]}} N_{k(\alpha, \beta, \gamma)/k} (\beta - \gamma, \beta - \alpha) \\
\rm{and} \quad & N_{k_{v_h} /k} \rho_2 \partial_{v_h} \{f(X), g(X), h(X) \} = {\frac {\deg f(X) \deg g(X) \deg h(X)} {[k(\alpha, \beta, \gamma):k]}} N_{k(\alpha, \beta, \gamma)/k} (\gamma - \alpha, \gamma - \beta).
\end{align*}

Therefore, by Proposition \ref{Milnor-map} and by the fact that the transfer map is just a multiplication by the degree $[k(\alpha, \beta, \gamma):k]$
of the field extension for the elements contained in the base field, it suffices to show that
$\{\alpha - \beta, \alpha - \gamma \} + \{\beta - \gamma, \beta - \alpha \} + \{\gamma - \alpha, \gamma - \beta \} = \{-1, -1 \}$
in $K^M_2(k(\alpha, \beta, \gamma))$.
But, since $\displaystyle {\frac {\alpha-\beta} {\gamma-\beta}} + {\frac {\alpha-\gamma} {\beta-\gamma}} = 1$, we have
$\displaystyle \{{\frac {\alpha-\beta} {\gamma-\beta}} , {\frac {\alpha-\gamma} {\beta-\gamma}} \} =0$.
Hence, $\displaystyle \{\alpha-\beta, \alpha-\gamma \} - \{\alpha-\beta, \beta-\gamma \} - \{\gamma-\beta, \alpha-\gamma\} =0$
and we are done since $-\{\alpha-\beta, \beta-\gamma \} = \{\beta-\gamma, \alpha-\beta \} = \{\beta-\gamma, \beta-\alpha \} + \{\beta-\gamma, -1\}$,
$- \{\gamma-\beta, \alpha-\gamma\} = \{\alpha-\gamma, \gamma-\beta\} = \{\gamma-\alpha, \gamma-\beta\} + \{-1, \gamma-\beta\}$, and
$\{\beta-\gamma, -1\} + \{-1, \gamma-\beta\} = \{\beta-\gamma, -1\} - \{\gamma-\beta, -1\} = \{-1, -1\}$.

For $l \ge 3$, if we go through a similar argument which is only notationally more complicated,
the proof boils down to the computation of the following element in the Milnor's $K$-group:
$(-1)^l \{ \vartheta_0 - \vartheta_1,\vartheta_0 - \vartheta_2, \dots, , \vartheta_0 - \vartheta_{l-1}, \vartheta_0 - \vartheta_{l} \}
+ \{ \vartheta_1 - \vartheta_2, \vartheta_1 - \vartheta_3, \dots, \vartheta_1 - \vartheta_{l}, \vartheta_1 - \vartheta_0 \}
+ (-1)^l \{ \vartheta_2 - \vartheta_3, \vartheta_2 - \vartheta_4, \dots, \vartheta_2 - \vartheta_{0}, \vartheta_2 - \vartheta_1 \}
+ \dots + \{\vartheta_l - \vartheta_0, \vartheta_l - \vartheta_1, \dots, \vartheta_l - \vartheta_{l-2}, \vartheta_l - \vartheta_{l-1} \}$,
where none of $\vartheta_i$ ($i=0,\dots,l$) and their differences are 0. Note that the signs for the ($l+1$)-terms are
all plus if $l$ is even and alternating if $l$ is odd. We claim that this expression is equal to $\{-1, \dots, -1 \}$ in $K^M_l(L)$, where
$L=k(\vartheta_0, \dots, \vartheta_l)$.

We regard the indices modulo $l+1$ and write $x_i = \vartheta_0 - \vartheta_{i}$.
Then the $i$-th term ($i=0,1,\dots, l$) in the above expression, if we disregard signs, becomes
$\{x_{i+1}-x_i, x_{i+2}-x_{i}, x_{i+3}-x_i, \dots, x_{i+l}-x_i \}$.

Therefore, the proof is complete by Proposition \ref{Milnor-key-relation}.
\end{proof}

The following key result shows the compatibility between these two types of transfer maps.

\begin{prop} \label{norm-compatible}
For every finite field extension $k \subset L$, we have the following commutative diagram, where the vertical maps are the transfer maps
and the horizontal maps are the homomorphisms in Proposition \ref{Milnor-map}:
$$\xymatrix{K^M_l (L) \ar[r]^-{\rho_l} \ar[d]^-{N_{L/k}} &
    GW_l(L) \ar[d]^-{N_{L/k}} \\
    K^M_l (k) \ar[r]^-{\rho_l} &
    GW_l(k)}$$
\end{prop}
\begin{proof}
Because of the functoriality properties of the transfer maps for both Milnor's $K$-groups and Goodwillie groups,
we may assume that $L=k(\alpha)$ is a simple extension of $k$.

We need to prove that, for each generator $x=\partial_v(y) \in K^M_l(k_v)$ as in Lemma \ref{generator-residue-field}, we have an equality
$N_{k_v/k}(\rho_l(x)) = \rho_l(N_{k_v/k}(x))$ in $GW_l(k)$.
For such a generator $x \in K^M_l(k_v)$, by Lemma \ref{norm-inductive-formula},
we may express $N_{k_v/k} (x) = N_{k_v/k} (\partial_v(y))$ as a sum of $\pm \{-1, -1, \dots, -1\}$ and
$N_{k(v_i)/k} (x_i)$ in $K^M_l(k)$, for $i=1,2,\dots,r$, where $\deg \pi_{v_i} < \deg \pi_v$ and $x_i \in K^M_l(k_{v_i})$ for each $i$.
Note also that, for such a generator $x=\partial_v(y)$, $\partial_w(y)$ vanishes if $\deg(w) \ge \deg(v)$ unless $w$ is equal to $v$ or $v_\infty$.
Thanks to Definition \ref{transfer-Milnor} and Lemma \ref{Weil-like-formula}, we have
$\displaystyle \sum_{v} \rho_l \Big( N_{k_v /k} \big( \partial_v (y) \big) \Big)=0$ and
$\displaystyle \sum_{v} N_{k_v /k} \Big( \rho_l \big( \partial_v (y) \big) \Big) = 0$.
So, the proof of our proposition is complete by induction on the degree of $L/k$ as the proposition holds trivially when $L=k$.
\end{proof}

The following lemma essentially gives a procedure to associate Milnor symbols to a given tuple of commuting matrices and will be used to construct
the inverse to the map $\rho_l$ in Proposition \ref{Milnor-map}.

\begin{lemma} \label{symbol-splits}
For any field $k$ and commuting matrices $A_1, \dots, A_l \in GL_n(k)$, there exist finite field extensions $L_1, \dots, L_r$ of $k$ and
$\alpha_{ij} \in GL_1(L_j) = L_j^\times$ ($1 \le i \le l$, $1 \le j \le r$)
such that $\displaystyle \sum_{j=1}^r N_{L_i/k} (w_j) = (A_1, \dots, A_l)$ in $GW_l(k)$,
where $w_j = (\alpha_{1j}, \dots, \alpha_{lj}) \in GW_l(L_j)$ ($j=1,\dots, r$). Moreover,
the choices of the fields $L_i$ ($i=1, \dots, r$)
and the elements $\alpha_{ij} \in L_j^\times$ can be made canonically so that it depends only on given $A_1, \dots, A_l$.
\end{lemma}

\begin{proof}
Let us write $z = (A_1, A_2, \dots, A_l)$, where $A_1, A_2, \dots, A_l$ are commuting matrices
in $GL_n(k)$. We then consider the vector space $E=k^n$ as a $k[t_1, \dots, t_l]$-module, on which $t_i$ acts as $A_i$. Since $E$
is of finite rank over $k$, it has a composition series $0=E_0 \subset E_1 \subset \dots \subset E_r=E$
with simple factors $L_j = E_j / E_{j-1}$ ($j=1,\dots,r$).

For each $j$, there exists a maximal ideal $\mathfrak{m}_j$ of $k[t_1, \dots, t_l]$ such that $L_j \simeq k[t_1, \dots, t_l] / \mathfrak{m}_j$.
So we can see that $L_j$ is a finite extension of $k$, and $\displaystyle z = \sum_{j=1}^r (A_1 | L_j, \dots, A_l | L_j)$,
where $A_i |L_j$ is the invertible linear map from $L_j$ onto itself induced by $A_i$.

Let us denote by $\alpha_{ij}$ the element of $L_j^\times$ which corresponds to $t_i$ (mod $\mathfrak{m}_j$) for $i=1, \dots, l$, then
$\displaystyle (A_1 | L_j, \dots, A_l | L_j) = N_{L_j/k} \left((\alpha_{1j}, \dots, \alpha_{lj})\right)$.
Take $w_j = (\alpha_{1j}, \dots, \alpha_{lj})$ and we are done.

Finally, note that the fields $L_i$ ($i=1, \dots, r$) and the elements $\alpha_{ij} \in L_j^\times$ do not change except for the order if we choose a different
composition series, e.g., by Jordan H\"{o}lder theorem.
\end{proof}

Now, we show that Milnor's $K$-groups are isomorphic to Goodwillie groups, which is similar to
a theorem by Nesterenko-Suslin in \cite{MR992981} which states that Milnor's $K$-gruops and certain Bloch's higher Chow groups are isomorphic

\begin{theorem} \label{Milnor-iso}
For any field $k$, the assignment $\{ a_1, \dots, a_l \} \mapsto (a_1,\dots,a_l) $ for each Milnor symbol
$\{a_1, \dots, a_l \}$ gives an isomorphism $\displaystyle \rho_l: K^M_l(k) {\overset{\sim}{\rightarrow}} GW_l(k)$ for $l \ge 1$.
In fact, it gives rise to a graded ring isomorphism between $\displaystyle \bigoplus_{l \ge 0} K^M_l(k)$ and $\displaystyle \bigoplus_{l \ge 0} GW_l(k)$.
\end{theorem}
\begin{proof}
The case $l=1$ is done in Proposition \ref{GW1-field}. By Proposition \ref{Milnor-map}, the assignment $\{ a_1,a_2, \dots, a_l \} \mapsto (a_1,a_2,\dots,a_l) $
gives rise to a homomorphism $\rho_l$ from the Milnor's $K$-group $K^M_l(k)$ to the Goodwillie group $GW_l(k)$.
We will construct the inverse map $ \phi_l : GW_l(k) \rightarrow K^M_l(k)$ for $l \ge 2$ as follows.

For each commuting matrices $A_1, \dots, A_l$ in $GL_n(k)$, by Lemma \ref{symbol-splits},
we may canonically find finite field extensions $L_1, \dots, L_r$ of $k$ and elements $\alpha_{ij}$ of $L_j^\times$
($1 \le i \le l$, $1 \le j \le r$) such that
$\displaystyle (A_1, \dots, A_l) = \sum_{j=1}^r N_{L_j/k} \big((\alpha_{1j}, \dots, \alpha_{lj})\big)$ in $GW_l(k)$.
We set
$$ \phi_l (A_1, \dots, A_l) = \sum_j N_{L_j/k}\left( \{ \alpha_{1j}, \dots, \alpha_{lj} \} \right),$$
where $\{ \alpha_{1j}, \dots, \alpha_{lj} \}$ is a Milnor symbol and $N_{L_j/k}: \ K^M_l (L_j) \rightarrow K^M_l (k)$
is the transfer map for the Milnor's $K$-groups.
By Lemma \ref{symbol-splits}, $\phi_l$ gives a map from the set of $l$-tuples of commuting matrices into $K^M_l(k)$.
Then, it is immediate that $\rho_l(\phi_l (A_1, \dots, A_l)) = (A_1, \dots, A_l)$ in $GW_l(k)$ by Proposition \ref{norm-compatible}.
In particular, $\rho_l$ is surjective.

To show that $\phi_l$ actually gives a map from $GW_l(k)$ into $K^M_l(k)$,
it remains to show that $\phi_l$ vanishes on the relations $(i)$,$(ii)$,$(iii)$ and $(iv)$ in Definition \ref{GW_l}.
But, $\phi_l$ clearly vanishes on the relations of type $(i)$, $(ii)$ and $(iii)$
and thus $\phi_l$ gives rise to a homomorphism from $C(k, l)$ onto $K^M_l(k)$.
To verify it for the relation $(iv)$, let $A_1(X), \dots, A_l(X)$ be commuting matrices in $GL_n(k[X])$, where $X$ is an indeterminate.
Then $M=k(X)^n$ can be considered as $k(X)[t_1, \dots, t_l]$-module, on which $t_i$ acts as $A_i(X)$. Then
find a composition series $0=M_0 \subset M_1 \subset \dots \subset M_r=M$ with simple factors $Q_j = M_j / M_{j-1}$ ($j=1,\dots,r$)
and maximal ideals $\mathfrak{n}_j$ of $k(X)[t_1, \dots, t_l]$ such that $Q_j \simeq k(X)[t_1, \dots, t_l] / \mathfrak{n}_j$.
We also denote by $\beta_{ij}$ the element of $Q_j^\times$ which corresponds to $t_i$ (mod $\mathfrak{n}_j$) for $i=1, \dots, l$ and $j=1, \dots, r$.
Each $Q_j$ is a finite extension field of $k(X)$ and let $x = \sum_{j=1}^r N_{Q_j / k(X)} (\beta_{1j}, \dots, \beta_{lj}) \in K^M_l(k(X))$.

Now consider $y=\{x, {(X-1) / X} \}$ in $K^M_{l+1}(k(X))$,
where $\{x, {(X-1) / X} \}$ is a symbol which represents $\sum_u \{x_{1u}, \dots, x_{lu}, {(X-1) / X} \}$ if $x = \sum_u \{x_{1u}, \dots, x_{lu}\}$ in $K^M_l(k(X))$.
Using the product formula for Milnor's $K$-groups (\cite{MR689382}) and by the fact that the monic irreducible polynomials of $\beta_{ij}$ over $k(X)$
have their coefficients in $k[X]$ and thus constant terms in $k \subset k(X)$,
we may deduce that $x$ can be written as a sum of Milnor symbols whose coordinates are in $k \subset k(X)$.
Consequently, the image $\partial_v (y)$ is zero unless $v$ is the valuation associated with either $\pi_v = X-1$ or $\pi_v = X$.
Also note that $\partial_v (y) = \phi_l \big( (A_1(0), \dots, A_l(0)) \big)$ if $\pi_v = X$ and
$\partial_v (y) = \phi_l \big( (A_1(1), \dots, A_l(1)) \big)$ if $\pi_v = X-1$. By Definition \ref{transfer-Milnor},
we see that $\phi_l \big( (A_1(0), \dots, A_l(0)) \big) = \phi_l \big( (A_1(1), \dots, A_l(1)) \big)$. Therefore, $\phi_l$ gives
a well-defined map from $GW_l(k)$ onto $K^M_l(k)$.

Finally, we note that $\phi_l \circ \rho_l$ is also the identity map on $K^M_l(k)$ since each Milnor symbol is fixed by it and the proof is complete.
\end{proof}
\begin{corollary}
Every element of $GW_l(k)$ can be written as $(A_1, \dots, A_l)$, where each $A_i$ is a diagonal matrix in $GL_n(k)$ for some $n \ge 1$.
\end{corollary}
\begin{proof}
This statement is an immediate consequence of the fact that $\rho_l$ in Theorem \ref{Milnor-iso} is surjective.
\end{proof}

\section{Joint determinants} \label{sec-joint-det}

We first state a suitable set of axioms for joint determinants.

\begin{definition} \label{joint-determinant}
A joint determinant $D\, (=D_l)$ ($l \ge 1$) is a map from the set of $l$-tuples of commuting matrices in $GL_n(k)$ ($n \ge 1$) into some abelian group $(G,+)$
which satisfies the following properties. \\
(i) (Multilinearity) For $l+1$ commuting matrices $A_1, \dots, A_l$ and $B$ in $GL_n(k)$ for some $n \ge 1$, we have
$D(A_1, \dots, A_i B, \dots, A_l) = D(A_1, \dots, A_i, \dots, A_l) + D(A_1, \dots, B, \dots, A_l)$.\\
(ii) (Block Diagonal Matrices) For commuting $A_1, \dots, A_l \in GL_m(k)$ and commuting $B_1, \dots, B_l \in GL_n(k)$ for some
$m, n \ge 1$, we have
$$D \left( \begin{pmatrix} A_1 & 0 \\ 0 & B_1 \end{pmatrix}, \dots, \begin{pmatrix} A_l & 0 \\ 0 & B_l \end{pmatrix} \right)
= D(A_1, \dots, A_l) + D(B_1, \dots, B_l).$$
(iii) (Similar Matrices) For commuting matrices $A_1, \dots, A_l \in GL_n(k)$ and any $S \in GL_n(k)$,
we have $D(S A_1 S^{-1}, \dots, S A_l S^{-1}) = D(A_1, \dots, A_l)$.\\
(iv) (Polynomial Homotopy) For commuting $A_1(t), \dots, A_l(t) \in GL_n(k[t])$, we have $D(A_1(0), \dots, A_l(0)) = D(A_1(1), \dots, A_l(1))$.
\end{definition}

For example, the usual determinant is a joint determinant with $l=1$ and $G=k^\times$, the multiplicative group of units in $k$.

We also note that $(iii)$ and $(iv)$ in Definition \ref{joint-determinant} are automatically satisfied
if we require $D$ to satisfy the following much stronger condition: \\
$(iii')$ (Compatibility with the Usual Determinant) For commuting matrices $A_1, \dots, A_l \in GL_n(k)$ and commuting matrices $B_1, \dots, B_l \in GL_n(k)$
such that $\det A_i = \det B_i$ for $i=1, \dots, l$, we have $D(A_1, \dots, A_l) = D(B_1, \dots, B_l)$.

The following theorem states that the map from the set of $l$-tuples of commuting invertible matrices to the Milnor's $K$-group, which is defined
via the inverse to $\rho_l$ in Theorem \ref{Milnor-iso}, is the universal one among all possible joint determinants.
\begin{theorem} \label{Milnor-joint-determinant}
There is a one-to-one correspondence between the set of joint determinants from the set of $l$-tuples of commuting invertible matrices
into an abelian group $G$ and the set of group homomorphisms from $K^M_l(k)$ into $G$.
\end{theorem}
\begin{proof}
We first note that the condition $(i)$ in Definition \ref{joint-determinant} can be replaced by the following condition: \\
$(i')$ (Identity Matrices) For commuting matrices $A_1, \dots, A_l \in GL_n(k)$ with $A_i = I_n$ for some $i$, we have $D(A_1, \dots, A_l)=0$.

$(i)$ clearly implies $(i')$ and conversely, if a map $D$ from the set of $l$-tuples of commuting invertible matrices into $G$
satisfies $(ii)$, $(iii)$ and $(iv)$ in Definition \ref{joint-determinant} and $(i')$ above, then $D$ gives rise to a homomorphism
from $GW_l(k)$ into $G$ by Definition \ref{GW_l}. Then Corollary \ref{multilin-l-l} shows that $D$ should satisfy $(i)$ in Definition \ref{joint-determinant}.

Now the one-to-one correspondence between these two sets is immediate from Definition \ref{GW_l} and Theorem \ref{Milnor-iso}.
In particular, a joint determinant is defined via $\phi_l$ in the proof of Theorem \ref{Milnor-iso} followed by a homomorphism from $K^M_l(k)$ into $G$.
\end{proof}

Now we investigate the case where the target abelian group $G$ is equal to the multiplicative group $k^\times$ as in case with the usual determinant.
We refer \cite{MR90c:18010} for an overview of explicit computations of some Milnor's $K$-groups.

\begin{corollary} \label{determinant-R}
For $l \ge 2$, there exists only one nontrivial joint determinant $D_l$ from the set of $l$-tuples of commuting invertible matrices over $\R$ into $\R^\times$,
which is continuous when restricted to the set of commuting matrices in $GL_n(\R)$, for each $n$, with the usual topology.
In this case, we have $D_l(-1, -1, \dots, -1)= -1$ and the image of $D_l$ is equal to $\{ \pm 1 \}$.
\end{corollary}
\begin{proof}
By \cite{MR0260844}, we have $K^M_l(\R) \simeq (\Z/2) \oplus H$, where the first direct factor $\Z/2$ is generated by $\{-1, \dots, -1\}$ and $H$ is a uniquely divisible group.
Then any homomorphism from $H$ into $\R^\times$ under our interest should be trivial by Theorem A.1. of \cite{MR0349811} when $l=2$ and the proof of it
can be generalized easily for the case $l > 2$.
Therefore, the only nontrivial continuous joint determinant is given by the projection
$K^M_l(\R) \rightarrow \Z/2$ followed by $\Z/2 {\overset{\sim}{\rightarrow}}  \{ \pm 1 \} \subset \R^\times$.
Note that, when $l=2$, the only nontrivial continuous joint determinant is essentially the isomorphism $\phi_l : GW_l(\R) {\overset{\sim}{\rightarrow}} K^M_l(\R)$
followed by the Hilbert symbol $(\ ,\ )_\R$.
\end{proof}

\begin{corollary} \label{determinant-Q}
For $l=2$, there are countably many nontrivial joint determinants from the set of pairs of commuting invertible matrices over $\Q$
into $\Q^\times$. In all cases, the image of a nontrivial joint determinant is $\{ \pm 1 \}$. For $l \ge 3$, there is only one nontrivial joint determinant
from the set of $l$-tuples of commuting invertible matrices over $\Q$ into $\Q$.
\end{corollary}
\begin{proof}
By Theorem 11.6 of \cite{MR0349811}, we have $\displaystyle K^M_2(\Q) \simeq \left( \Z/2 \right) \bigoplus_{p {\rm \ prime}} (\Z/p)^\times$.
Since it is a torsion group, we immediately notice that any nontrivial joint determinant
from the set of pairs of commuting invertible matrices over $\Q$ into $\Q^\times$ should be mapped onto $\{\pm 1\}$.
The first direct factor is generated by the Milnor symbol $\{-1, -1 \}$ and the projection $K^M_2(\Q) \rightarrow \Z/2$ followed by an isomorphism
$\Z/2 {\overset{\sim}{\rightarrow}} \{\pm 1\}$ is nothing but the Hilbert symbol $(\ ,\ )_\R$.
For each direct factor $(\Z/p)^\times$, we have a unique homomorphism onto $\{\pm 1 \}$ which is given by the quadratic residue
$\displaystyle \left({\frac { } {p}} \right)$ and it corresponds to the Hilbert symbol $(\ ,\ )_{\Q_p}$ for each prime $p$
(See Theorem 4.4.9 of \cite{MR1282290}).
Therefore, any homomorphism from $K^M_l(\Q)$ onto $\{\pm 1\}$ is a finite product of these Hilbert symbols and there are countably many possibly choices.

For $l \ge 3$, we have $K^M_l(\Q) \simeq \Z/2$ (See \cite{MR0442061}) and the proof is complete.
\end{proof}

\begin{corollary} \label{determinant-finite-field}
For a finite field $k$ and $l \ge 2$, there exists no nontrivial joint determinant $D_l$ from the set of $l$-tuples of commuting invertible matrices over $k$ into $k^\times$.
\end{corollary}
\begin{proof}
This follows from the fact that $K^M_l(k) \simeq 0$ when $k$ is a finite field and $l \ge 2$ (Example 1.5 of \cite{MR0260844}).
\end{proof}

 \bibliographystyle{plain}

\end{document}